\documentclass[11pt]{article}
\usepackage{amsmath, amssymb, amsthm, enumerate}

\usepackage{graphicx}

\newtheorem{theorem}{Theorem}[section]
\newtheorem{lemma}[theorem]{Lemma}
\newtheorem{corollary}[theorem]{Corollary}

\def\ints{{\mathbb Z}}
\def\cx{{\mathbb C}}
\def\cA{{\mathcal A}}
\def\cB{{\mathcal B}}
\def\cH{{\mathcal H}}
\def\cN{{\mathcal N}}
\def\cS{{\mathcal S}}
\def\yy#1#2{Y_{{#1},{#2}}} 
\def\yab{\yy ab}
\def\spn{\mathop{\mathrm{span}}\nolimits}

\def\snv{^{(-)}}
\def\snt{^{(-)T}}
\def\inv{^{-1}}
\def\Th{\Theta}
\def\th{\theta}

\def\thw{\Th_W}

\def\nom#1{\cN_{#1}}
\def\nomw{\nom{W}}

\def\one{{\bf 1}}
\def\qn1{{q^{n-1}}}

\title{Hamming Graphs in Nomura Algebras}
\author{
Ada Chan
\thanks{The first author gratefully acknowledges the
support from a JSPS postdoctoral fellowship and a NSERC
discovery grant.}\\
Department of Mathematics and Statistics\\
York University, Toronto, Ontario, Canada M3J 1P3\\
ssachan@yorku.ca\\
\and
Akihiro Munemasa\\
Graduate School of Information Sciences\\
Tohoku University,
Sendai, 980-8579, Japan\\
munemasa@math.is.tohoku.ac.jp
}
\begin{document}
\maketitle

\begin{abstract}
Let $\cA$ be an association scheme on $q\geq 3$ vertices.
We show that the Bose-Mesner algebra of 
the generalized Hamming scheme $\cH(n,\cA)$, for $n\geq 2$,
is not the Nomura algebra of any type~II matrix.

This result gives examples of formally self-dual Bose-Mesner algebras 
that are not the Nomura algebras of type~II matrices.
\end{abstract}

\section{Introduction}
A $v \times v$ matrix $W$ is a {\em type~II matrix} if 
\begin{equation}
\label{defTypeII}
\sum_{x=1}^v \frac{W(x,a)}{W(x,b)}=\delta_{a,b} v
\quad \text{for $a, b = 1,\ldots,v$.}
\end{equation}
Hadamard matrices and the character tables of finite abelian groups satisfy 
this condition.  Type II matrices also arise from combinatorial objects such
as symmetric designs, tight sets of equiangular lines and strongly regular graphs
\cite{TypeIIComb}.  

In \cite{MR99f:05125}, Nomura constructed the Bose-Menser algebra of an association scheme
from each type~II matrix $W$, hence another connection to combinatorics.
We call this algebra the {\em Nomura algebra of $W$}, and denote it by $\nomw$.
Jaeger, Matsumoto and Nomura \cite{MR99f:05125} 
showed that a type~II matrix $W$
belongs to $\nomw$ if and only if $cW$ is a spin model
for some non-zero scalar $c$.
Spin models give link invariants and they are difficult to find.
The ability to identify the Bose-Mesner algebras that are
the Nomura algebras of type~II matrices 
is a step towards the search of spin models.

We say two type~II matrices $W_1$ and $W_2$ are {\em type~II equivalent} if
\begin{equation*}
W_1 = P_1 D_1 W_2 D_2 P_2
\end{equation*}
for some invertible diagonal matrices 
$D_1$ and $D_2$ and permutation matrices $P_1$ and $P_2$.
Suzuki \cite{Suzuki} showed that if $\nomw$ contains the Hamming scheme $\cH(n,3)$, then
$W$ is type~II equivalent to the character table of the group $\ints_3^n$.

It is straightforward to show that $W_1$ and $W_2$ are type~II matrices
if and only if their Kronecker
product $W_1 \otimes W_2$ is a type~II matrix.
We generalize Suzuki's result and show that if $\nomw$ contains the 
adjacency matrix of the Hamming graph $H(n, q)$,
$n\geq 2$ and $q \geq 3$, then $W$ is type~II equivalent to 
\begin{equation*}
W_1 \otimes W_2 \otimes \cdots \otimes W_n
\end{equation*}
where $W_1, W_2, \ldots, W_n$ are $q\times q$ type~II matrices.
In this case, $\nomw$ is isomorphic to
\begin{equation*}
\nom{W_1} \otimes \nom{W_2} \otimes \cdots \otimes \nom{W_n}.
\end{equation*}
As a consequence, when $n\geq 2$ and $q\geq 3$,
the Bose-Mesner algebra of 
the Hamming scheme $\cH(n,q)$ and 
the Bose-Mesner algebra of the generalized Hamming scheme
$\cH(n,\cA)$, for any association scheme $\cA$ on $q$ vertices,
cannot be the Nomura algebras of type~II matrices.

If $W$ is type~II, then so is $W^T$.  In \cite{MR99f:05125},
Jaeger et al.\ showed that $\nomw$ and $\nom{W^T}$ are 
formally dual.  
If $\nomw=\nom{W^T}$, which includes the case where $W$ is symmetric, then
$\nomw$ is formally self-dual.  
The Hamming scheme $\cH(n,q)$ is a formally self-dual association scheme.
It also satisfies Delsarte's notion \cite{MR0384310}
of self-duality and the notion of hyper self-duality defined by Curtin and
Nomura \cite{MR1826951}.
We now have a family of association schemes satisfying 
all three notions of duality but are
not the Nomura algebras of type~II matrices.

In this paper, we use $I_v$ and $J_v$ to denote the $v\times v$ identity
matrix and
the matrix of all ones, respectively.  We use ${\bf 1}_v$ to denote the column vector of all
ones of length $v$.  We omit the subscript if the size or the length is clear.

\section{Nomura Algebras}
\label{SectionTypeII}

We recall type~II matrices and their Nomura algebras,  
see \cite{MR99f:05125} and \cite{MR98g:05158} for details. 

Given matrices $M$ and $N$ of the same order, their {\em Schur product},
$M\circ N$, is defined by
\begin{equation*}
(M \circ N)(i,j) = M(i,j)N(i,j)
\qquad \text{for all $i$ and $j$.}
\end{equation*}
If $M$ has no zero entry, we use $M\snv$ to denote the matrix whose $(i,j)$-entry is $M(i,j)^{-1}$.

A $v\times v$ matrix $W$ is type~II if 
$W\snt W=vI$.
The $v\times v$ matrix
\begin{equation*}
W = (t-1)I+J
\end{equation*}
is type~II if and only if $t$ is a root of the quadratic
$t^2+(v-2)t+1=0$.
This is the simplest spin model (after an
appropriate scaling), called the Potts model, and 
the associated link invariant is the Jones polynomial \cite{MR89m:57005}.

Let the vectors $e_1, \ldots, e_v$ be the standard basis of $\cx^v$.
Given a $v\times v$ type~II matrix $W$, we define $v^2$ vectors
\begin{equation*}
\yab = We_a \circ W\snv e_b
\qquad
\text{for $a, b = 1, \ldots, v$.}  
\end{equation*}
It follows from (\ref{defTypeII}) that $W$ is invertible and has no zero
entries.  Hence, for all $a$, the set
$\{\yab : b=1,\ldots,v\}$ is a basis for $\cx^v$.

The {\em Nomura algebra}, $\nomw$, of $W$ is the set of matrices
that have $\yab$ as eigenvectors for all $a, b$.   It follows immediately that 
$\nomw$ contains $I$ and it is closed under matrix multiplication.
By (\ref{defTypeII}), $J \yab= \delta_{a,b} v \yab$.
Hence $J \in \nomw$ and $\dim\nomw \geq 2$.

For each matrix $M\in \nomw$, let $\thw(M)$ be the $v\times v$ matrix whose
$(a,b)$-entry satisfies
\begin{equation*}
M \yab = \thw(M)(a,b) \yab
\qquad
\text{for $a, b=1, \ldots, v$.}   
\end{equation*}

The following results, obtained from Theorems~1 to 3 of {\cite{MR99f:05125}},
 are useful in subsequent sections.
\begin{theorem}
\label{ThmNom}
Let $W$ be a $v\times v$ type~II matrix.
Then $\nomw$ is closed under matrix multiplication, Schur product and
transpose.  It is commutative with respect to matrix multiplication.

Moreover $\thw(\nomw) = \nom{W^T}$, and
\begin{enumerate}
\item
$\thw(M_1M_2)=\thw(M_1)\circ\thw(M_2)$,
\item
$\thw(M_1\circ M_2) = \frac{1}{v} \thw(M_1)\thw(M_2)$,
\item
$\thw(M_1^T) =  \thw(M_1)^T$,
\end{enumerate}
for $M_1, M_2 \in \nomw$.
\qed
\end{theorem}

\begin{lemma}
\label{LemmaEquiv}
Let $W$ be a type~II matrix.  For any invertible diagonal matrices $D_1$ and $D_2$
and for any permutation matrices $P_1$ and $P_2$,
\begin{equation*}
\nom{P_1D_1 W D_2P_2} = P_1 \nomw P_1^T.  
\end{equation*}
\qed
\end{lemma}

An {\em association scheme} on $v$ elements with $n$ classes is a set of $v\times v$ $01$-matrices
\begin{equation*}
\cA = \{ A_0, A_1, \ldots, A_n\}
\end{equation*}
satisfying
\begin{enumerate}
\item
$A_0=I$.
\item
$\sum_{i=0}^n A_i = J$.
\item
$A_i^T \in \cA$ for $i=0,\ldots, n$.
\item
$A_iA_j$ lies in the span of $\cA$ for $i, j = 0, \ldots, n$.
\item
$A_iA_j=A_jA_i$ for $i, j=0, \ldots, n$.
\end{enumerate}
The simplest association scheme is $\{I, J-I\}$, called the {\em trivial 
association scheme}.

The span of an association scheme over $\cx$ is called its {\em Bose-Mesner algebra}.
The Bose-Mesner algebra of $\cA$ is closed under matrix multiplication, Schur product and transpose.  
It is commutative with respect to matrix multiplication, and it contains $I$ and $J$.   
Conversely, any algebra satisfying these conditions is the Bose-Mesner algebra of an association
scheme \cite{MR1002568}.

A {\em formal duality} between two Bose-Mesner algebras $\cB_1$ and $\cB_2$
is an invertible linear map $\Th:\cB_1\rightarrow\cB_2$ satisfying
$\Th(MN)=\Th(M)\circ\Th(N)$ and $\Th(M\circ N)=\frac{1}{v}\Th(M)\Th(N)$.
When $\cB_1=\cB_2$ and $\Th^2(M)=vM^T$, we say $\cB_1$ is 
{\em formally self-dual}.

\begin{theorem}
If $W$ is a type~II matrix, then $\nomw$
and $\nom{W^T}$ are a formally dual pair of
Bose-Mesner algebras.
\qed
\end{theorem}

\section{Products}
Suppose $W_1$ and $W_2$ are type~II matrices then so is their Kronecker product $W_1\otimes W_2$.
Proposition~7 of {\cite{MR99f:05125}} determines the Nomura algebra of
$W_1\otimes W_2$.
\begin{lemma}
\label{LemmaTensor}
If $W_1$ and $W_2$ are type~II matrices, then
\begin{equation*}
\nom{W_1 \otimes W_2} = \nom{W_1} \otimes \nom{W_2}.
\end{equation*}
\qed
\end{lemma}

Hosoya and Suzuki 
\cite{MR1958007}
studied the structure of $W$ if $J \otimes I$ belongs to $\nomw$.
In \cite{Suzuki},  Suzuki showed that if 
$I\otimes J$ and $J\otimes I$ belong to $\nomw$, 
then $W$ is type~II equivalent to the Kronecker product of
two type~II matrices.
We tailor Theorem~1.2 of \cite{Suzuki} to Lemma~\ref{LemSuzuki} and
Theorem~\ref{ThmProduct}
by showing the type~II equivalence explicitly, for later use.

For the remainder of this section, we assume that
$W$ is an $mn\times mn$ type~II matrix and
\begin{equation*}
I_n \otimes J_m, J_n\otimes I_m \in \nomw.
\end{equation*}
For a vector $u$ of length $mn$, we use $u[i]$ to denote its $i$-th block of length $m$ for $i=1,\dots,n$.
We also denote by $W[i,j]$ the $m\times m$ submatrix located
at the $(i,j)$-block of $W$.

\begin{lemma}
\label{LemmaEquiv12}
For $1\leq a,b\leq mn$, the following statements hold.
\begin{enumerate}
\item
$\thw(I_n\otimes J_m)(a,b)=m$ if and only if
\begin{equation*}
\yab[i]\in\spn(\one_m) \quad \text{for all $i=1,\dots,n$,} 
\end{equation*}
or equivalently,
$We_a[i]\in\spn(We_b[i])$ for all $i=1,\dots,n$.
\item
$\thw(J_n\otimes I_m)(a,b)=n$ if and only if
\begin{equation*}
\yab[i]=\yab[1] \quad \text{for all $i=1,\dots,n$,}
\end{equation*}
or equivalently,
\begin{equation*}
(We_a[i])\circ(W\snv e_a[1])=(We_b[i])\circ(W\snv e_b[1])
\quad \text{for all $i=1,\dots,n$.}
\end{equation*}
\end{enumerate}
\end{lemma}
\begin{proof}
Straightforward.
\end{proof}

\begin{lemma}\label{LemmaWP}
There exists a permutation matrix $P$ such that 
\begin{equation}\label{Eqn_WP}
\Th_{WP}(I_n \otimes J_m) = m (J_n\otimes I_m)
\quad \text{and} \quad
\Th_{WP}(J_n \otimes I_m) = n (I_n\otimes J_m).
\end{equation}
\end{lemma}
\begin{proof}
For $1\leq a,b \leq mn$, we write $a\sim_1 b$ when
$\thw(I_n\otimes J_m)(a,b)=m$,
and $a\sim_2 b$ when $\thw(J_n\otimes I_m)(a,b)=n$.
Since $\{\yy{a}{1},\ldots,\yy{a}{mn}\}$ is a basis of $
\cx^{mn}$
consisting of eigenvectors of $I_n\otimes J_m$
which has eigenvalue $m$ with multiplicity $n$, 
$|\{b\mid a\sim_1 b\}|=n$, for each given $a$.
It follows from Lemma~\ref{LemmaEquiv12} that
$\sim_1$ is an equivalence relation
with $m$ equivalence classes of size $n$.
Similarly, for each given $a$,
$|\{b\mid a\sim_2 b\}|=m$, and
Lemma~\ref{LemmaEquiv12} implies that
$\sim_2$ is an equivalence relation with $n$ equivalence classes of size $m$.

Furthermore, Lemma~\ref{LemmaEquiv12} implies that
$a\sim_1 b$ and $a\sim_2 b$ 
occur simultaneously
only when $a=b$.
Hence an equivalence class of $\sim_1$ meets every equivalence class of $\sim_2$
in exactly one element, and
there exists a permutation matrix $P$ so that
the equivalence classes of $\sim_1$ and $\sim_2$ defined for $WP$ are
\begin{equation*}
\{h, m+h, 2m+h, \ldots, (n-1)m+h\}
\quad
\text{for $h=1, \ldots, m$} 
\end{equation*}
and 
\begin{equation*}
\{rm+1, rm+2, \ldots, rm+m \}
\qquad
\text{for $r=0, \ldots, n-1$,}
\end{equation*}
respectively.
By Lemma~\ref{LemmaEquiv}, $\nomw = \nom{WP}$.
We conclude that $\nom{WP}$ contains $I_n\otimes J_m$ and $J_n\otimes I_m$,
and (\ref{Eqn_WP}) holds.
\end{proof}

We say a type~II matrix is {\sl normalized} if all entries in its first row and
its first column are $1$.
Given any type~II matrix $W$, there exists invertible diagonal matrices $D$
and $D'$ such that $W'=DWD'$ is normalized.
By Lemma~\ref{LemmaEquiv}, we have $\nom{W'}=\nomw$.
Note that the eigenvector $W'e_a \circ {W'}\snv e_b$ is a scalar multiple
of $We_a\circ W\snv e_b$.
We conclude that $\Th_{W'}(M) = \thw(M)$ for all $M\in \nomw$.
\begin{lemma}
\label{LemSuzuki}
Suppose $W$ is normalized
and
\begin{equation*}
\thw(I_n \otimes J_m) = m (J_n\otimes I_m)
\quad \text{and} \quad
\thw(J_n \otimes I_m) = n (I_n\otimes J_m).
\end{equation*}
Then $W = U \otimes V$
for some $n\times n$ type~II matrix $U$ and some $m\times m$ type~II matrix $V$.
\end{lemma}
\begin{proof}
By Lemma~\ref{LemmaEquiv12},
we have $a \equiv b\pmod m$ if and only if
\begin{equation*}
We_a[i] \in \spn(We_b[i])  \quad \text{for $i=1,\ldots,n.$}
\end{equation*}
Since the first row of $W$ is $\one^T_{mn}$,  setting $i=1$ gives
\begin{equation}\label{Eqn_W111n}
W[1,1]=W[1,2]=\ldots=W[1,n].
\end{equation}
Let $V=W[1,1]$.
Further, the first column of $W$ is $\one_{mn}$, so
there exists a non-zero scalar $U(i,j)$ such that 
\begin{equation}
\label{Eqn_Tensor}
We_{(j-1)m+1}[i]=U(i,j)\one_m \quad \text{for $i,j=1,\ldots,n$.}
\end{equation}

Suppose $a=(j-1)m+h$ and $c=(j-1)m+1$ for $j=1,\ldots,n$
and $h=1,\ldots,m$.
As $\thw(J_n \otimes I_m) = n (I_n\otimes J_m)$,
Lemma~\ref{LemmaEquiv12} implies that
\begin{align*}
We_a[i]&=
(Ve_h)\circ(U(i,j)\one_m)\circ(W\snv[1,1]e_1)
&&\text{(by (\ref{Eqn_W111n}), (\ref{Eqn_Tensor}))}
\\ &=
U(i,j)(Ve_h).
\end{align*}
Hence $W[i,j]=U(i,j)V$ for $i,j=1,\ldots,n$.
It is straightforward to check that both $U$ and $V$ are type~II.
\end{proof}

\begin{theorem}
\label{ThmProduct}
Suppose $W$ is a type~II matrix and
\begin{equation*}
I_n \otimes J_m, J_n\otimes I_m \in \nomw.
\end{equation*}
Then 
\begin{equation*}
W=D_1(U\otimes V) D_2 P
\end{equation*} 
for some $n\times n$ type~II matrix $U$, $m\times m$ type~II matrix $V$,
permutation matrix $P$, and invertible diagonal matrices $D_1$ and $D_2$.
In this case, 
\begin{equation*}
\nomw = \nom{U} \otimes \nom{V}.
\end{equation*}
\end{theorem}
\begin{proof}
By Lemma~\ref{LemmaWP}, 
there exists a permutation matrix $P$ such that 
(\ref{Eqn_WP}) holds.
There exist invertible diagonal matrices $D$
and $D'$ such that $W'=DWPD'$ is normalized.
Then by Lemma~\ref{LemSuzuki},
$W' = U \otimes V$
for some $n\times n$ type~II matrix $U$ and 
$m\times m$ type~II matrix $V$.
The rest of the proof is immediate from
Lemma~\ref{LemmaEquiv} and Lemma~\ref{LemmaTensor}.
\end{proof}

\section{Generalized Hamming Schemes}
\label{SectionGenHam}

We recall from \cite{GodsilGenHam} the definition of and some facts
concerning the generalized Hamming scheme $\cH(n,\cA)$.
Let $\cA=\{A_0,A_1,\ldots,A_d\}$ be an association scheme on $q$ vertices.
Consider the product association scheme $\cA^{\otimes n}$
and the symmetric group $\cS_n$ acting on $\{1,\ldots,n\}$.  
For each element $\sigma \in \cS_n$, define
\begin{equation*}
(A_{i_1}\otimes A_{i_2}\otimes \cdots \otimes A_{i_n})^{\sigma}
= A_{i_{1\sigma\inv }} \otimes
 A_{i_{2\sigma\inv }} \otimes \cdots\otimes
 A_{i_{n\sigma\inv }}.
\end{equation*}
Then $\cS_n$ is a group of algebra automorphism of 
the span of $\cA^{\otimes n}$.
The set of matrices in the span of
$\cA^{\otimes n}$ fixed by every element of $\cS_{n}$
is closed under matrix multiplication, Schur product and transpose,
and this set contains $I_{q^n}$ and $J_{q^n}$.
It is the Bose-Mesner algebra of a subscheme of $\cA^{\otimes n}$ 
\cite{GodsilGenHam}.  
This subscheme is called the {\em generalized Hamming scheme} $\cH(n,\cA)$.  
In particular, for $i=1,\ldots,d$, the matrix
\begin{equation*}
(A_i\otimes I_q \otimes \cdots \otimes I_q) +
(I_q\otimes A_i \otimes \cdots \otimes I_q) + \cdots +
(I_q\otimes I_q \otimes \cdots \otimes A_i) 
\end{equation*}
lies in $\cH(n,\cA)$.
The {\sl Hamming scheme} $\cH(n,q)$ is $\cH(n,\cA)$ when $\cA$ is the trivial
association scheme on $q$ vertices.

Let $\Omega$ be the set of words of length $n$ over an alphabet of size $q$.
The Hamming graph $H(n,q)$ has vertex set $\Omega$, and two words are adjacent if and only if
they differ in exaclty one position.  
We use $A(n)$ to denote the adjacency matrix of $H(n,q)$.
Up to permutation of the vertices, we can write $A(n)$ as
\begin{align}
\label{EqnRecursion}
&\left[(J_q-I_q)\otimes I_q \otimes \cdots \otimes I_q\right] +
\left[I_q\otimes (J_q-I_q) \otimes \cdots \otimes I_q\right] 
\\ \nonumber
&\quad+ \cdots+
\left[I_q\otimes \cdots \otimes I_q \otimes (J_q-I_q)\right] \\
&= 
\sum_{i=1}^d
\left[
(A_i\otimes I_q \otimes \cdots \otimes I_q) +
(I_q\otimes A_i \otimes \cdots \otimes I_q) 
\right.
\nonumber\\ &\quad+ \left.\cdots+
(I_q\otimes I_q \otimes \cdots \otimes A_i) \right].
\nonumber
\end{align}
Therefore $A(n)$ lies in the Bose-Mesner algebra of $\cH(n,\cA)$ for any association scheme $\cA$ on $q$ vertices.
More importantly, $A(n)$ satisfies
the recursion
\begin{eqnarray*}
A(n) &=& (J_q-I_q)\otimes I_\qn1 + I_q \otimes A(n-1)\\
&=&
\begin{pmatrix}
A(n-1) & I_\qn1 & \cdots & I_\qn1\\
I_\qn1 & A(n-1) &  \cdots & I_\qn1\\
\vdots & \vdots & \ddots & \vdots\\
I_\qn1 & I_\qn1 &  \cdots & A(n-1)
\end{pmatrix}.
\end{eqnarray*}

Here are some facts about $A(n)$ that are useful in the next section,
see \cite{MR1002568} and \cite{MR1220704} for details.
The matrix $A(n)$ has $n+1$ eigenvalues
\begin{equation*}
\th_h(n) = (q-1)(n-h) - h \quad
\text{for $h= 0, \ldots, n$}.
\end{equation*}
The eigenspace of $\th_h(n)$, denoted by $V_h(n)$, has dimension $(q-1)^h\binom{n}{h}$.
Note that $\th_0(n)=(q-1)n$ is the valency of the vertices in the Hamming graph $H(n,q)$, so
$\one_{q^n}$ is an eigenvector of $A(n)$ belonging to the
eigenvalue $\th_0(n)$.  Since $V_0(n)$ has dimension one,
$V_0(n) =\spn(\one_{q^n})$.

The next lemma exhibits the recursive nature of the eigenvectors of $A(n)$ in $V_h(n)$ when $h\geq 1$.  
Given a column vector $u$ of length $q^n$,
we use $u[i]$ to denote the $i$-th block of $u$ of length $\qn1$.
\begin{lemma}
\label{LemEigenvectors}
Let $1\leq h\leq n$. Then $u\in V_h(n)$ if and only if
\begin{equation}
u[i]-u[j]\in  V_{h-1}(n-1)
\quad \text{for $i,j=1,\ldots, q$,}\label{Eqn2}\\
\end{equation}
and
\begin{equation}
\sum_{i=1}^q u[i]\in  V_h(n-1).\label{Eqn3}
\end{equation}
In particular, $u\in V_1(n)$ if and only if
there exist a vector $w\in V_1(n-1)$
and scalars $a_1, \ldots, a_q$ such that
$a_1+\dots+a_q=0$
and
\begin{equation*}
u[i]=w+a_i\one_\qn1
\quad \text{for $i=1,\ldots, q$.}
\end{equation*}
\end{lemma}
\begin{proof}
From
\begin{equation*}
\begin{pmatrix}
A(n-1) & I_\qn1 &  \cdots & I_\qn1\\
I_\qn1 & A(n-1) &  \cdots & I_\qn1\\
\vdots & \ddots &  \vdots & \vdots\\
I_\qn1 & I_\qn1 & \cdots & A(n-1)
\end{pmatrix}
\begin{pmatrix}
u[1]\\
u[2]\\
\vdots\\
u[q]
\end{pmatrix}
=
\th_h(n)
\begin{pmatrix}
u[1]\\
u[2]\\
\vdots\\
u[q]
\end{pmatrix},
\end{equation*}
we get
\begin{equation}
\label{Eqn1}
\left(A(n-1)-I_\qn1\right) u[i] +\sum_{l=1}^q u[l] = 
\th_h(n) u[i]
\end{equation}
for $i=1,\ldots, q$.
It follows that
\begin{equation*}
\left(A(n-1)-I_\qn1\right) (u[i]-u[j]) = \th_h(n)(u[i]-u[j])
\end{equation*}
or 
\begin{equation*}
A(n-1) (u[i]-u[j]) = \th_{h-1}(n-1)(u[i]-u[j])
\end{equation*}
for $i, j = 1,\ldots, q$ and (\ref{Eqn2}) follows.

We also get from (\ref{Eqn1}) that
\begin{equation*}
\sum_{i=1}^q\left((A(n-1)-I_\qn1) u[i] +\sum_{l=1}^q u[l]\right) = 
\th_h(n) \sum_{i=1}^q u[i]
\end{equation*}
which leads to
\begin{equation*}
A(n-1)\sum_{i=1}^q u[i] =
\th_h(n-1) \sum_{i=1}^q u[i].
\end{equation*}
Hence (\ref{Eqn3}) is true.
The converse is straightforward.

Suppose $u\in V_1(n)$.
From (\ref{Eqn2}), there exist 
scalars $a_1, \ldots, a_q$ such that
\begin{equation*}
\frac{1}{q}\sum_{j=1}^q(u[i]-u[j])=a_i\one_\qn1
\quad \text{for $i=1,\ldots,q$.}
\end{equation*}
Set 
\begin{equation*}
w=\frac1q \sum_{j=1}^q u[j].
\end{equation*}
Then 
by (\ref{Eqn3}), $w\in V_1(n-1)$ and
$w+a_i\one_\qn1=u[i]$ holds for $i=1,\dots,q$.
Since
\begin{equation*}
\sum_{i=1}^q a_i \one_\qn1
=\frac{1}{q}\sum_{i,j=1}^q(u[i]-u[j])=0,
\end{equation*}
we see that $a_1+\cdots+a_q=0$.
The converse is again straightforward.
\end{proof}

\section{When $\nomw$ contains the Hamming graph}
\label{SectionResult}
In this section, we assume $W$ is a type~II matrix and $A(n)$ is the adjacency matrix of
the Hamming graph $H(n,q)$ given in (\ref{EqnRecursion}) for some $n\geq 2$ and $q\geq 3$.
\begin{lemma}
Suppose $A(n) \in \nomw$.  If $\yab\in V_1(n)$, then either
\begin{equation*}
\yab = 
\begin{pmatrix}
a_1 \one_\qn1 &
a_2 \one_\qn1 &
\cdots &
a_q \one_\qn1
\end{pmatrix}^T
\end{equation*}
where $a_1+a_2+\cdots+a_q=0$, or
\begin{equation*}
\yab = 
\begin{pmatrix}
w & w & \cdots & w 
\end{pmatrix}^T
\end{equation*}
for some non-zero vector $w \in V_1(n-1)$.
\end{lemma}
\begin{proof}
From Lemma~\ref{LemEigenvectors}, there exist $w\in V_1(n-1)$ and 
scalars $a_1, \ldots, a_q$ satisfying $a_1+\cdots + a_q=0$
such that
\begin{equation}
\label{Form}
\yab[i] = w + a_i \one_\qn1
\quad \text{for $i=1,\ldots,q$}.
\end{equation}

Now suppose $w$ is not the zero vector and not all $a_i$'s are zero,
and we shall derive a contradiction.
By Theorem~\ref{ThmNom} and the symmetry of $A(n)$,
\begin{equation*}
\thw(A(n))(a,b) = \thw(A(n))(b,a)
\end{equation*}
so $Y_{b,a} \in V_1(n)$.
Similar to $\yab$, it follows from (\ref{Eqn2}) that
there exist scalars $c_{ij}$ such that
\begin{equation}
\label{Eqn4}
Y_{b,a}[i]-Y_{b,a}[j]=\yab[i]\snv - \yab[j]\snv = c_{ij} \one_\qn1
\end{equation}
for all $i, j =1, \ldots, q$.

Applying (\ref{Eqn4}) to the $r$-th and the $s$-th blocks gives
\begin{equation}\label{Eqn9}
\frac{1}{w(l)+a_r} - \frac{1}{w(l)+a_s} = c_{rs} \quad
\text{for $l=1,\ldots, \qn1$.}
\end{equation}
There exists $r\in\{1,\dots,q\}$ such that $a_r\neq0$,
and since $a_1+\cdots+a_q=0$, there exists 
$s\in\{1,\dots,q\}$ such that $a_s\neq a_r$. Then
$c_{rs}\neq0$ by (\ref{Eqn9}).
Hence, for $l=1,\ldots, \qn1$, $w(l)$ is a root of the quadratic
\begin{equation}\label{EqnQ}
x^2+(a_r+a_s)x+a_ra_s+\frac{a_r-a_s}{c_{rs}}=0.
\end{equation}
Since $w\in V_1(n-1)$ is orthogonal to $\one_\qn1$ and $w\neq0$,
there exist $l$ and $l'$ such that $w(l)\neq w(l')$.
Then $w(l)$ and $w(l')$ are the roots of the
quadratic (\ref{EqnQ}).
This implies that $w$ has two distinct entries, the sum
of which is $-(a_r+a_s)$. Also, since $s$ was arbitrary
subject to $a_s\neq a_r$, we see that 
the $a_i$'s take exactly two distinct values.  

Let $w$ have $x$ entries equal $\frac{-(a_r+a_s)}{2} + \alpha$
and $(\qn1-x)$ entries equal $\frac{-(a_r+a_s)}{2}-\alpha$ where
$\alpha\neq0$.
Let $\yab$ have $y$ blocks equal $w+a_r \one_\qn1$, $q-y$ blocks equal 
$w+a_s\one_\qn1$.
Since $\yab, Y_{b,a} \in V_1(n)$, we have
\begin{eqnarray*}
\one_{q^n}^T \yab = &(2y-q)\qn1\frac{(a_r-a_s)}{2} +(2x-\qn1)q\alpha& = 0,\\
\one_{q^n}^T Y_{b,a} = &\frac{1}{(\frac{a_r-a_s}{2})^2-\alpha^2}
\left((2y-q)\qn1\frac{(a_r-a_s)}{2} -(2x-\qn1)q\alpha \right) & = 0
\end{eqnarray*}
These two equations give 
\begin{equation*}
x=\frac{\qn1}{2} \quad \text{and} \quad y=\frac{q}{2}.
\end{equation*}
This is a contradiction if $q$ is odd.

Now assume $q$ is even.
Then 
$\sum_{i=1}^q a_i= \frac{q}{2} a_r + \frac{q}{2}a_s=0$,
so $a_s=-a_r$.  
Assume, without loss of generality, that the first $\frac{q}{2}$ blocks 
of $\yab$ are $w+a_r\one_\qn1$ and the last $\frac{q}{2}$ blocks are $w-a_r\one_\qn1$.

Since $\dim V_1(n)=(q-1)n>1$ and $\{Y_{b,c}:c\in\Omega\}$
is a basis of $\cx^{q^n}$, there exists $c\neq a$ such that
$Y_{b,c} \in V_1(n)$.
From (\ref{Eqn2}), there exist scalars $b_{ij}$ such that 
\begin{equation*}
Y_{b,c}[i]-Y_{b,c}[j]=b_{ij}\one_\qn1
\qquad
\text{for $i,j=1,\ldots,q$.}
\end{equation*}
There exists $k\in\{0,1\dots,n\}$ such that
$Y_{a,c}\in V_k(n)$.
Then by (\ref{Eqn2}), we have
\begin{equation*}
Y_{a,c}[i]-Y_{a,c}[j] \in V_{k-1}(n-1).
\end{equation*}
On the other hand,
\begin{equation*}
Y_{a,c}[i]-Y_{a,c}[j] = 
\begin{cases}
b_{ij}(w+a_r\one_\qn1)
&\text{if $1\leq i,j\leq\frac{q}{2}$,}\\
b_{ij}(w-a_r\one_\qn1)
&\text{if $\frac{q}{2}\leq i,j\leq q$.}
\end{cases}
\end{equation*}
Since $w$ and $a_r\one_\qn1$ are non-zero vectors in distinct eigenspaces of $A(n-1)$,
we have $b_{ij}=0$ for $1 \leq i, j \leq \frac{q}{2}$
and for $\frac{q}{2}+1\leq i, j\leq q$.
Therefore the first $\frac{q}{2}$ blocks of $Y_{b,c}$ are identical and
the last $\frac{q}{2}$ blocks of $Y_{b,c}$ are identical.
If we let $u=Y_{b,c}[1]+Y_{b,c}[q]$, then
\begin{equation*}
Y_{b,c}[i]=
\begin{cases}
\frac{b_{1q}}{2}\one_\qn1 + \frac{1}{2} u & \text{if $i=1,\ldots,\frac{q}{2}$,}\\
-\frac{b_{1q}}{2}\one_\qn1 + \frac{1}{2} u & \text{if $i=\frac{q}{2}+1,\ldots,q$.}
\end{cases}
\end{equation*}
By (\ref{Eqn3}),
\begin{equation*}
\sum_{i=1}^q Y_{b,c}[i] = \frac{q}{2} u \in V_1(n-1).
\end{equation*}
So if $Y_{b,c} \in V_1(n)$, then it lies in
the span of
\begin{equation*}
\left\{
\begin{pmatrix} 
\one_\qn1\\
\vdots\\
\one_\qn1\\
-\one_\qn1\\
\vdots\\
-\one_\qn1\\
\end{pmatrix}
\right\}
\cup
\left\{
\begin{pmatrix} 
u\\
\vdots\\
u\\
u\\
\vdots\\
u\\
\end{pmatrix}
: u \in V_1(n-1)
\right\},
\end{equation*}
which has dimension at most $1+(q-1)(n-1)$.
The set $\{Y_{b,c}: c\in \Omega\}$ is a basis of $\cx^{q^n}$, so there should be
$\dim V_1(n)=(q-1)n$ eigenvectors of the form $Y_{b,c}$ in
$V_1(n)$.  But when $q\geq 3$, $(q-1)n > 1+(q-1)(n-1)$.
This is a contradiction.
\end{proof}

\begin{lemma}
If $A(n)\in \nomw$, then 
\begin{equation*}
I_q \otimes J_\qn1 \in \nomw.
\end{equation*}
\label{LemmaA_1N_w}
\end{lemma}
\begin{proof}
Suppose $\yab \in V_0(n)$.  Then $\yab[i] \in \spn(\one_\qn1)$ for $i=1,\ldots,q$, and
$\yab$ is an eigenvector of $I_q \otimes J_\qn1$ 
belonging to the eigenvalue $\qn1$.

Suppose $\yab \in V_1(n)$.  By the previous lemma, either all of
$\yab[1],\ldots,\yab[q]$ lie in $\spn(\one_\qn1)$ or they all lie in $V_1(n-1)$.
In the former case, $\yab$ is an eigenvector of $I_q\otimes J_\qn1$ 
belonging to the eigenvalue $\qn1$.
In the latter case, $\yab$ is an eigenvector of $I_q\otimes J_\qn1$
belonging to the eigenvalue $0$.

Suppose $\yab \in V_h(n)$ for some $h > 1$.   It follows from (\ref{Eqn2}) that
\begin{equation*}
J_\qn1(\yab[i]-\yab[j])={\bf 0}
\end{equation*}
for all $1 \leq i,j\leq q$.
From (\ref{Eqn3}), we have
\begin{equation*}
\sum_{i=1}^q J_\qn1\yab[i]={\bf 0}.
\end{equation*}
These two equations give
\begin{equation*}
J_\qn1\yab[i]={\bf 0}
\qquad
\text{for $i=1,\ldots,q$.}
\end{equation*}
Therefore $\yab$ is an eigenvector of $I_q \otimes J_\qn1$ 
belonging to the eigenvalue $0$.
\end{proof}

\begin{theorem}
If $A(n) \in \nomw$, then $W$ is type~II equivalent to $W_1 \otimes \cdots \otimes W_n$ and 
\begin{equation*}
\nomw = \nom{W_1} \otimes \cdots\otimes \nom{W_n},
\end{equation*}
where $W_1, \ldots, W_n$ are $q\times q$ type~II matrices.
\label{ThmA_1N_w}
\end{theorem}
\begin{proof}
By Lemma~\ref{LemmaA_1N_w},
we have $I_q \otimes J_\qn1 \in \nomw$.  Then
\begin{equation*}
J_q\otimes I_\qn1
=A(n) - 
\left(A(n)\circ(I_q \otimes J_\qn1)\right)
+ I_{q^n} 
\end{equation*}
also belongs to $\nomw$.

Theorem~\ref{ThmProduct} tells us that $W$ is type~II equivalent to 
$W_1 \otimes V$ for some $q\times q$ type~II matrix 
$W_1$ and $\qn1\times \qn1$ type~II matrix $V$, and
\begin{equation*}
\nomw = \nom{W_1}\otimes \nom{V}.
\end{equation*}
Observe that 
\begin{equation*}
A(n)\circ \left(I_q\otimes J_\qn1\right) = I_q\otimes A(n-1) \in \nomw,
\end{equation*}
so $A(n-1) \in \nom{V}$.  The theorem follows by induction.
\end{proof}

We now give Theorem~1.3 of \cite{Suzuki} 
as an immediate consequence of this theorem and the fact that 
the unique $3\times 3$ type~II matrix, up to type~II equivalence, is
\begin{equation}
\label{EqnSize3}
\begin{pmatrix}
1 & 1 & 1\\
1 & \omega & \omega^2\\
1 & \omega^2 & \omega
\end{pmatrix},
\end{equation}
where $\omega$ is a primitive cube root of unity.
\begin{corollary}
If $\cH(n,3) \subseteq \nomw$, then $W$ is type~II equivalent to a character table of 
$\ints_3^n$.
\qed
\end{corollary}

\begin{theorem} 
Let $\cA$ be an association scheme on $q\geq 3$ vertices.
Then, for $n\geq 2$,
the Bose-Mesner algebra of $\cH(n,\cA)$ is not the Nomura algebra of
a type~II matrix.
\end{theorem}
\begin{proof}
Suppose that the Bose-Mesner algebra of $\cH(n,\cA)$ coincides
$\nomw$ for some type~II matrix $W$.
Since $A(n)$ belongs to the span of $\cH(n,\cA)$,
it follows from Theorem~\ref{ThmA_1N_w} that
\begin{equation*}
\nomw=\nom{W_1} \otimes \cdots \otimes \nom{W_n},
\end{equation*}
where $W_1, \ldots, W_n$ are $q\times q$ type~II matrices.
There exists a Schur idempotent $A_1\neq I$ of
$\cN_{W_1}$, and 
$A_1\otimes I\otimes\cdots\otimes I$
belongs to the association scheme defined by $\nomw$.
This forces
$A_1\otimes I\otimes\cdots\otimes I\in\cH(n,\cA)$
which is absurd.
\end{proof}

It is known that if $\cA$ is formally self-dual, then so is
$\cH(n,\cA)$ \cite{GodsilGenHam}.  The corollary gives plenty of examples
of formally self-dual association schemes that are not the Nomura algebras
of type~II matrices.

\begin{corollary}
The Bose-Mesner algebra of $\cH(n,q)$, $n\geq 2$ and $q\geq 3$, 
is not the Nomura algebra of a type~II matrix.
\qed
\end{corollary}

When $n=1$, $\cH(1,q)$ is the trivial scheme on $q$ vertices.
It follows from Theorem~6.4 of \cite{TypeIIComb} that
the Nomura algebra of the
Potts model of size $q$, for $q\geq 5$, is trivial.
The Nomura algebra of the $3\times 3$ type~II matrix in (\ref{EqnSize3})
has dimension three.  
The Nomura algebra of a $4\times 4$
type~II matrix has dimension at least three \cite{MR99f:05125}.
So the Bose-Mesner algebra of $\cH(1,q)$ is the Nomura algebra of
a type~II matrix exactly when $q\geq 5$.

The Bose-Mesner algebra of $\cH(2,2)$ is the Nomura algebra of 
\begin{equation*}
\begin{pmatrix}
1 & 1 & 1 & 1\\
1 & 1 & -1 & -1\\
1 & -1 & \alpha& -\alpha\\
1 & -1 & -\alpha & \alpha
\end{pmatrix}
\end{equation*}
when $\alpha$ is not a fourth root of unity
\cite{MR99f:05125}.

\section*{Acknowledgement}
The authors would like to thank Hiroshi Suzuki for bringing our attention 
to this problem and
his result for $q=3$ in \cite{Suzuki}, and many useful discussions.
\bibliography{Hamming}
\bibliographystyle{acm}

\end{document}